\documentclass[a4paper,11pt,english]{amsart}

\usepackage{babel,varioref}
\usepackage[ansinew]{inputenc}
\usepackage{amsmath}
\usepackage{amsfonts}
\usepackage{amscd}
\usepackage{enumerate}
\usepackage{graphics}
\usepackage{babel,varioref}
\usepackage{amscd}
\usepackage[pdftex]{graphicx}
\DeclareGraphicsExtensions{.pdf,.png,.jpg}

\input amssym.def
\input amssym.tex

\def\reels{\mathbb{R}}

\def\corps{\mathbb{K}}

\def\im{\mbox{\textnormal{im}}}

\def\g {\mathfrak}

\newtheorem{theo}{Theorem}
\newtheorem{prop}[theo]{Proposition}
\newtheorem{corr}[theo]{Corollary}
\newtheorem{lem}[theo]{Lemma}
\newtheorem{defi}{Definition}
\newtheorem{ex}{Example}

\title{Representations admitting two pairs of supplementary invariant spaces}
\author[Lionel Bérard Bergery]{Lionel Bérard Bergery}
\author[Tom Krantz]{Tom Krantz}
\thanks{e-mail: Lionel.Berard-Bergery@iecn.u-nancy.fr, Tom.Krantz@uni.lu\\
Address: Institut Élie Cartan Nancy. Université Henri Poincaré Nancy
1.\\ B.P. 239, F-54506 Vandoeuvre-lès-Nancy Cedex, France.\\
The second author was supported by BFR 03/095 of the ministry of
research of the G.-D. of Luxembourg }

\makeindex
\begin{document}
\begin{abstract}
We examine the lattice generated by two pairs of supplementary
vector subspaces of a finite-dimensional vector-space by intersection
and sum, with the aim of applying the results to the study of
representations admitting two pairs of supplementary invariant
spaces, or one pair and a reflexive form. We show that such a
representation is a direct sum of three canonical sub-representations
which we characterize. We then focus on holonomy representations
with the same property.
\end{abstract}
\maketitle
\section{Introduction}

A famous paper of Gelfand and Ponomarev~\cite{gp} classifies the
systems on four vector subspaces of a finite-dimensional vector
space. We focus on the systems of two pairs of supplementary spaces
and explore the lattice generated by sum and intersection starting
from the four spaces. The aim is to apply the results to lattices of
stable spaces of finite-dimensional representations and in
particular of holonomy representations of torsion free connections
preserving a reflexive form.

\section{Lattice generated by two pairs of supplementary spaces}
We suppose throughout the paper that $\corps$ is a commutative field
of characteristic different from 2.
\subsection{Definitions}\index{paire d'espaces supplémentaires}
We call {\em decomposition of a finite-dimensional $\corps$-vector
space $E$ into $2$ direct sums} a quintuplet ${\mathcal V}= (E, V_1,
V_2, W_1, W_2)$ where $V_1, V_2, W_1$ and $W_2$ are four
vector subspaces of the finite-dimensional vector space $E$
verifying $V_1 \oplus V_2 = W_1 \oplus W_2=E$.

\begin{ex}
In particular if $E$ carries a non-degenerate reflexive
structure(i.e. for us a non-degenerate symmetric or antisymmetric
bilinear form) and if $E=V_1\oplus V_2$ then
$(E,V_1,V_2,V_1^\perp,V_2^\perp)$ is a decomposition of $E$ into $2$
direct sums.
\end{ex}

Associated to a decomposition of a finite-dimensional
$\corps$-vector space $E$ into $2$ direct sums ${\mathcal V}=(E,
V_1, V_2, W_1, W_2)$ is a dual decomposition into two direct sums:
${\mathcal V}^*=(E^*, W'_1, W'_2, V'_1, V'_2)$, with $X':=\{u\in E^*
\; \vrule \; u(X)=0 \;\}$.

If $E=E_1\oplus E_2$ is a direct sum, let $p_{E_1}^{E_2}$ be the
projection on $E_1$ parallely to $E_2$. To simplify notations lets
write $p_i$ for the projection on $V_i$ parallely to $V_{\tau(i)}$
and $q_i$ the projection on $W_i$ parallely to $W_{\tau(i)}$. We
define the map $\theta_{\mathcal V}: E\to E$ by $\theta_{\mathcal V}
= p_{W_1}^{W_2} \circ p_{V_1}^{V_2} - p_{V_1}^{V_2} \circ
p_{W_1}^{W_2}$. To simplify notations we write $\theta$ for
$\theta_{\mathcal V}$ if it is clear which $\mathcal V$ we mean.

It is easy to verify:
\begin{lem}
$\theta  =  p_{W_1}^{W_2} \circ p_{V_1}^{V_2} - p_{V_1}^{V_2} \circ
p_{W_1}^{W_2} =  p_{W_2}^{W_1} \circ p_{V_2}^{V_1} - p_{V_2}^{V_1}
\circ p_{W_2}^{W_1} = p_{V_2}^{V_1} \circ p_{W_1}^{W_2} -
p_{W_1}^{W_2} \circ p_{V_2}^{V_1} = p_{V_1}^{V_2} \circ
p_{W_2}^{W_1} - p_{W_2}^{W_1} \circ p_{V_1}^{V_2}$
\end{lem}

We have also:
\begin{lem}\label{swap}
$\theta(V_i) \subset V_{\tau(i)}$ and $\theta(W_i) \subset
W_{\tau(i)}$
\end{lem}

\begin{lem} If ${\mathcal V}^*$ is the dual system of $\mathcal V$ then
$\theta_{{\mathcal V}^*}=(\theta_{\mathcal V})^*$
\end{lem}
\begin{proof} We have:
\begin{eqnarray*} (\theta_{\mathcal V})^* & = &
(p_{W_1}^{W_2}\circ p_{V_1}^{V_2} -p_{V_1}^{V_2} \circ
p_{W_1}^{W_2})^*\\
& = & (p_{W_1}^{W_2}\circ p_{V_1}^{V_2})^* - (p_{V_1}^{V_2} \circ
p_{W_1}^{W_2})^*\\
& = & (p_{V_1}^{V_2})^*\circ (p_{W_1}^{W_2})^* -(p_{W_1}^{W_2})^*
\circ
(p_{V_1}^{V_2})^*\\
& = & p_{V'_2}^{V'_1}\circ p_{W'_2}^{W'_1} -p_{W'_2}^{W'_1} \circ
p_{V'_2}^{V'_1}\\
& = & \theta_{{\mathcal V}^*}
\end{eqnarray*}
\end{proof}

\subsection{Canonical decomposition of $E$}

\begin{defi} Let us define a sequence of vector subspaces of $E$:
$F(0):=\{0\}$, $F(n+1):=\sum_{i,j} ((F(n)+V_i) \cap (F(n)+W_j))$ for
$n\ge 0$.
\end{defi}

$(F(n))_n$ is an increasing sequence of vector subspaces of the
finite-dimensional vector space $E$ and necessarily stationary Let
us write $F$ or $F(\infty)$ the space $\sum_n F(n)$. $F$ is the
smallest fix-point of the increasing mapping $X \mapsto \sum_{i,j}
((X+V_i) \cap (X+W_j))$, and $F$ is the smallest common fix-point of
the four increasing mappings $X \mapsto (X+V_i) \cap (X+W_j)$ for
$i,j\in \{1,2\}$.

\begin{lem}\label{F1sd}
$F(1)=\bigoplus_{i,j} V_i \cap W_j$
\end{lem}
\begin{proof} By definition we have $F(1)=\sum_{i,j} V_i \cap W_j$, and it
is easy to see that the sum is necessarily direct.
\end{proof}

\begin{defi}
Let us define a sequence of vector subspaces of $E$: $\tilde F(0):=E$
$\tilde F(n+1):=\bigcap_{i,j} ((\tilde F(n) \cap V_i) + (\tilde F(n)
\cap W_j))$ for $n\ge 0$.
\end{defi}

$(\tilde F(n))_n$ if a decreasing sequence of vector subspaces of the
finite-dimensional vector-space $E$ and so stationary. Let$\tilde
F(\infty)$ or simply $\tilde F$ be the space $\bigcap_n \tilde
F(n)$. $\tilde F$ is the biggest fix-point of the decreasing mapping
$X \mapsto \bigcap_{i,j} ((X \cap V_i) + (X \cap W_j))$, and $\tilde
F$ is the biggest common fix-point of the four decreasing mappings $X
\mapsto (X \cap V_i) + (X \cap W_j)$, for $i,j\in\{1,2\}$.

\begin{prop}\label{kerimtheta} For every non-negative integer $n$
\begin{enumerate}
\item $\ker \theta^n=F(n)$ \item $\im \theta^n=\tilde F(n)$
\end{enumerate}
\end{prop}

\begin{proof}
\begin{enumerate}
\item Let us show first that $\ker \theta=F(1)=V_1 \cap W_1 + V_1 \cap
W_2 + V_2 \cap W_1 + V_2 \cap W_2$. If $x\in V_i \cap W_j$,
$\theta(x)=(-1)^{i+j}((p_i \circ q_j)(x) - (q_j \circ
p_i)(x))=(-1)^{i+j}(x - x) =0$. As $\theta$ is linear,
$\theta(F(1))=0$.

Inversely if $\theta(x)=0$, we have $(q_1 \circ p_1 - p_1 \circ
q_1)(x)=0$ and so $(q_1 \circ p_1)(x)=(p_1 \circ q_1)(x)$. We have
$(q_1 \circ p_1)(x)\in V_1 \cap W_1$. Similarly $(q_j \circ
p_i)(x)=(p_i \circ q_j)(x)$ an so $(q_j \circ p_i)(x)\in V_i \cap
W_j$. We deduce $x=q_1(x)+q_2(x)=\sum_{i,j} (q_j \circ p_i)(x)\in
\sum_{i,j} V_i \cap W_j=F(1)$.

\medskip

Let us show $F(n) \subset \ker \theta^n$. For $n=0$ it is clear. If
$n=k+1$, suppose $\ker \theta^k=F(k)$. Let $x\in F(n)=\sum_{i,j}
((F(k)+V_i) \cap (F(k)+W_j))$. $x$ can be written
$x_{11}+x_{22}+x_{12}+x_{21}$ with $x_{ij}\in ((F(k)+V_i) \cap
(F(k)+W_j))$. $x_{ij}= y_{ij} + z_{ij} = t_{ij} + u_{ij}$ with
$y_{ij},t_{ij}\in F(k)$, $z_{ij}\in V_i$ and $u_{ij}\in W_j$. We
have be induction hypothesis $\theta^k(y_{ij})=0$ and
$\theta^k(t_{ij})=0$. Be iterated application of lemma~\ref{swap} we
have $\theta^k(z_{ij})\in V_{\tau^k(i)}$ et $\theta^k(u_{ij})\in
W_{\tau^k(i)}$. As a consequence $\theta^k(x_{ij})\in
V_{\tau^k(i)}\cap W_{\tau^k(i)}$ and so $\theta^k(x)\in F(1)=\ker
\theta$, giving: $\theta^{k+1}(x)=0$.

Let us show $\ker \theta^n \subset F(n)$. For $n=0$, $\ker \theta^0=
\{0\}=F(0)$. For $n=k+1$, suppose $\ker \theta^k \subset F(k)$. Let
$x$ be such that $\theta^n(x)=0$. We have then
$\theta^k(\theta(x))=0$. By induction hypothesis $\theta(x)\in
F(k)$. So $(q_j \circ p_i)(x)-(p_i \circ q_j)(x)\in F(k)$ and as a
consequence: $(q_j \circ p_i)(x)\in (F(k)+V_i)$. As $(q_j \circ
p_i)(x)\in W_j$, $(q_j \circ p_i)(x)\in (F(k)+V_i)\cap W_j \subset
(F(k)+V_i)\cap (F(k)+W_j)$. Finally $x=\sum_{i,j} (q_j \circ p_i)(x)
\in \sum_{i,j} ((F(k)+V_i)\cap (F(k)+W_j))=F(n)$.

\medskip

\item





To show that $\im \theta^n=\tilde F(n)$, we will use
duality\footnote{We use the following lemma which is easy to show:
For $\Psi\in {\mathcal L}(E,F)$, $\ker \Psi^*=(\im \Psi)'$ and $\im
\Psi^*=(\ker \Psi)'.$ }:

In finite dimension it is easy to show by induction that for every
$n$, $(F_{\mathcal V}(n))'=\tilde F_{{\mathcal V}^*}(n)$ and
$(\tilde F_{\mathcal V}(n))'=F_{{\mathcal V}^*}(n)$.

So we have: $(\tilde F_{\mathcal V}(n))''=(F_{{\mathcal
V}^*}(n))'=(\ker \theta^n_{{\mathcal V}^*})'=(\ker
(\theta^*_{\mathcal V})^n)'=(\ker (\theta_{\mathcal V}^n)^*)'=(\im
\theta^n_{\mathcal V})''$. By injectivity in finite dimension of
$''$ we have $\im \theta^n_{\mathcal V}=\tilde F_{\mathcal V}(n)$.

\end{enumerate}
\end{proof}

\begin{prop}
\begin{enumerate}

\item $\forall n, F(n+1)=\theta^{-1}(F(n)),$

\item $\forall n, \tilde F(n+1)=\theta(\tilde F(n)).$
\end{enumerate}
\end{prop}
\begin{proof} We have:
$F(n+1)=\ker(\theta^{n+1})=\theta^{-1}(\ker(\theta^n))=\theta^{-1}(F(n))$
et $\theta(\tilde F(n))=\theta(\im(\theta^n))=\theta(\tilde
F(n))$.\end{proof}

From the first point one can deduce: $\forall n,
\theta(F(n+1))\subset F(n)$.

We recall without proof the following well known result:
\begin{prop}
If $E$ is a finite-dimensional vector space and $\Psi$ an
endomorphism of $E$ then the two subspaces of $E$: $E_N=\sum_n
\ker(\Psi^n)$ and $E_I=\bigcap_n \im(\Psi^n)$ are stable by $\Psi$
and we have $E=E_N \oplus E_I$. Moreover $\Psi\vrule_{E_N}$ is
nilpotent and $\Psi\vrule_{E_I}$ is invertible.
\end{prop}

The result applied to $E$ and the endomorphism $\theta$ gives us for
$F:=\sum_n F(n)$ and $\tilde F:=\bigcap_n \tilde F(n)$: $E=F\oplus
\tilde F$. Moreover $F$ and $\tilde F$ are stables by $\theta$ and
$\theta\vrule_F$ is nilpotent and $\theta\vrule_{\tilde F}$ is
invertible.

We say that the subspace $V$ of $E$ is {\em homogeneous}with respect
to the sum $E_1+ E_2$, where $E_1$ and $E_2$ are vector subspaces of
$E$ if: $V\cap(E_1+E_2)=(V\cap E_1)+(V\cap E_2)$. Similarly we say
that $V$ is {\em co-homogeneous} with respect to the intersection
$E_1\cap E_2$, if: $V+(E_1\cap E_2)=(V+E_1)\cap (V+E_2)$.

\begin{prop}\label{tildefv1v2}
\begin{enumerate}
\item $(\tilde F\cap V_1) \oplus (\tilde F\cap V_2)=\tilde F$
\item $(\tilde F\cap W_1) \oplus (\tilde F\cap W_2)=\tilde F$
\item $\forall i,j, (\tilde F\cap V_i) \oplus (\tilde F\cap
W_j)=\tilde F$
\end{enumerate}
\end{prop}
\begin{proof} Let us start by the proof of point 3. We have: $V_i
\cap W_j \subset F(1)$, which gives us $(\tilde F \cap
V_i)\cap(\tilde F \cap W_j)\subset \tilde F \cap F(1)=\{0\}$. From
$\tilde F=(\tilde F \cap V_i)+(\tilde F \cap W_j)$ we deduce then
$\tilde F=(\tilde F \cap V_i)\oplus (\tilde F \cap W_j)$.

Let us note $n_i=\dim(\tilde F\cap V_i)$ and $m_j:=\dim(\tilde F\cap
W_j)$. Point $3$ implies then that $n_i+m_j=\dim \tilde F$ (*). This
gives us $n_1=n_2$ and $m_1=m_2$.

As $V_1\cap V_2=\{0\}$, $(\tilde F \cap V_1)\cap(\tilde F \cap
V_2)=\{0\}$. As $(\tilde F \cap V_1)\oplus (\tilde F \cap V_2)
\subset \tilde F$, we have: $2 n_1=n_1+n_2\le \dim \tilde F$. (**)
Similarly $(\tilde F \cap W_1)\oplus (\tilde F \cap W_2) \subset
\tilde F$ et $2 m_1=m_1+m_2\le \dim \tilde F$. (***)

From (*),(**) and (***) follows that $2 n_i=2 m_j=\dim \tilde F$ and
that $(\tilde F \cap V_1)\oplus (\tilde F \cap V_2)= \tilde F$ and
$(\tilde F \cap W_1)\oplus (\tilde F \cap W_2)= \tilde F$.
\end{proof}

\medskip

We can refine the two first points of the proposition as follows:

\begin{prop}\label{tildefnv1v2}
For every non negative integer $n$ we have:
\begin{enumerate}
\item $(\tilde F(n)\cap V_1) \oplus (\tilde F(n)\cap V_2)=\tilde F(n)$
\item $(\tilde F(n)\cap W_1) \oplus (\tilde F(n)\cap W_2)=\tilde F(n)$
\end{enumerate}
\end{prop}
\begin{proof} We will just prove the first point, the proof of
the second point being similar.

By induction on $n$: For $n=0$ we have effectively: $\tilde
F(0)=E=V_1 \oplus V_2$. Suppose the the result true for $n$.
Evidently we have the inclusion: $(\tilde F(n+1)\cap V_1) \oplus
(\tilde F(n+1)\cap V_2) \subset \tilde F(n+1)$. Let $a\in \tilde
F(n+1)$. We can write $a=x+y$ with $x\in V_1$ and $y\in V_2$. Let us
show then $x,y\in\tilde F(n+1)$.

As $a\in \tilde F(n+1) \subset \tilde F(n)$ and $\tilde F(n)$ is
homogeneous with respect to $V_1 \oplus V_2$ we have: $x,y\in \tilde
F(n)$.

By definition of $\tilde F(n+1)$, $a$ we can write $a=x_{ij}+y_{ij}$
with $x_{ij}\in \tilde F(n) \cap V_i$ and $y_{ij}\in \tilde F(n)
\cap W_j$. We deduce that $x$ is an element of $\tilde
F(n+1)=\bigcap_{i,j} ((\tilde F(n) \cap V_i)+(\tilde F(n) \cap
W_j))$ by writing: $x=x+0=x+0=(x_{21}-y)+y_{21}=(x_{22}-y)+y_{22}$.
A similar reflection shows that $y\in \tilde F(n+1)$.\end{proof}

\medskip

We will see in the following that one can decompose canonically
$F(n)$.

Let's write $e=id_{\{1,2\}}$ and $\tau=(1 2)$ the elements of the
group ${\mathcal S}_2$ of the permutations of the set $\{1,2\}$. We
will write for $i=1,2$, $\bar i:=\tau(i)$. For $\sigma\in {\mathcal
S}_2$, we write $\bar \sigma$ the element of ${\mathcal S}_2$ such
that $\{\sigma, \bar \sigma\}={\mathcal S}_2$.

\begin{defi}
Let $F_\sigma(0)=0$ and $F_\sigma(n+1)=\sum_i ((F_\sigma(n)+V_i)\cap
(F_\sigma(n)+W_{\sigma(i)}))$.

One can see that $(F_\sigma(n))_n$ is an increasing sequence of
subvectorspaces of $E$, and so finally stationary (as $E$ is
finite-dimensional). Let's write $F_\sigma(\infty)$ or simply
$F_\sigma$ the space $\sum_n F_\sigma(n)$ {\em i.e.} the maximal
element of this sequence.
\end{defi}

Let's remark on the other hand that lemma~\ref{F1sd} implies that
$F_e(1)=(V_1\cap W_1) \oplus (V_2 \cap W_2)$, $F_\tau(1)=(V_1\cap
W_2) \oplus (V_1 \cap W_2)$ and $F(1)=F_e(1)\oplus F_\tau(1)$.

\begin{prop}\label{thetafs}
$\forall n, \theta(F_\sigma(n+1))\subset F_\sigma(n)$.
\end{prop}
\begin{proof} By induction: It is true for $n=0$. Suppose its
true up to order $n$. Let $x\in F_\sigma(n+1), y\in V_1, z\in
F_\sigma(n+1), t\in V_2, x'\in F_\sigma(n+1), y'\in W_1, z'\in
F_\sigma(n+1), t'\in W_2$, such that $x+y=x'+y'$ et $z+t=z'+t'$.

Let us show that $\theta(x+y+z+t)\in F_\sigma(n+1)$. Let us recall
first that $\theta(V_i)\subset V_{\tau(i)}$ and $\theta(W_j)\subset
W_{\tau(j)}$. We have consequently:
$\theta(x)+\theta(y)=\theta(x')+\theta(y')\in (F_\sigma(n)+V_2)\cap
(F_\sigma(n)+W_{\sigma(2)})$ and
$\theta(z)+\theta(t)=\theta(z')+\theta(t')\in (F_\sigma(n)+V_1)\cap
(F_\sigma(n)+W_{\sigma(1)})$. This gives us
$\theta(x+y+z+t)=\theta(x)+\theta(y)+\theta(z)+\theta(t)\in
F_\sigma(n+1)$. \end{proof}

We will need the following lemma:
\begin{lem}\label{AB}
Let $A_0, A, B_0, B$ be four vector subspaces of $E$ such that
$A_0\subset A$ et $B_0 \subset B$. We have then
$$(A+B_0)\cap(A_0+B)=A_0+B_0+(A\cap B).$$
\end{lem}
\begin{proof} The inclusion "$\supset$" is clear, as every
$A+B_0$, $A_0+B$ contains every $A_0$, $B_0$, $A\cap B$.

For the inclusion "$\subset$" let $x\in A$, $y_0\in B_0$, $x_0\in
A_0$, $y\in B$ such that $x+y_0=x_0+y$. One deduces $x-x_0=y-y_0 \in
A\cap B$. So $x+y_0=x_0+y_0+(x-x_0)\in A_0+B_0+(A\cap B).$
\end{proof}

\begin{prop}\label{Fscohomogene}
\begin{enumerate}
\item $F_\sigma(n)$ is co-homogeneous with respect to the direct sum
$V_1\oplus V_2$ or equivalently $(F_\sigma(n)+V_1)\cap
(F_\sigma(n)+V_2)=F_\sigma(n)$.

\item $F_\sigma(n)$ is co-homogeneous with respect to the direct sum
$W_1\oplus W_2$ or equivalently $(F_\sigma(n)+W_1)\cap
(F_\sigma(n)+W_2)=F_\sigma(n)$.
\end{enumerate}
\end{prop}
\begin{proof} We will prove the first point, the proof for the
second being similar.

By induction: For $n=0$ its clear. Suppose the result true at the
order $n$. It is evident that
$F_\sigma(n)\subset(F_\sigma(n)+V_1)\cap(F_\sigma(n)+V_2)$.

Let's prove the other inclusion: We have:

\begin{eqnarray*}F_\sigma(n+1)+V_1 & = & \sum_i
((F_\sigma(n)+V_i)\cap(F_\sigma(n)+W_{\sigma(i)}))+V_1\\ & \subset &
\underbrace{F_\sigma(n)+V_1}_A+\underbrace{(F_\sigma(n)+V_2)\cap(F_\sigma(n)+W_{\sigma(2)})}_{B_0}.
\end{eqnarray*}
Similarly
\begin{eqnarray*}F_\sigma(n+1)+V_2 & \subset &
\underbrace{(F_\sigma(n)+V_1)\cap(F_\sigma(n)+W_{\sigma(1)})}_{A_0}+\underbrace{F_\sigma(n)+V_2}_B.
\end{eqnarray*}
By application of lemma~\ref{AB} and the induction hypothesis we
obtain:
\begin{eqnarray*}(F_\sigma(n+1)+V_1)\cap (F_\sigma(n+1)+V_2) & \subset &
F_\sigma(n+1) +(F_\sigma(n)+V_1) \cap (F_\sigma(n)+V_2)\\
 & \stackrel{\mbox{\scriptsize ind. hyp.}}{\subset} & F_\sigma(n+1) + F_\sigma(n)\\
 & \subset & F_\sigma(n+1).
\end{eqnarray*}\end{proof}

\begin{prop}\label{Fsfbs1}
$\forall n, F_\sigma(n)\cap F_{\bar \sigma}(1)=\{0\}$.
\end{prop}
\begin{proof} Let's make the proof for $\sigma=e$, the case
$\sigma=\tau$ being analogous.

By induction: It is true up to order $n=0$. Suppose it is true up to
order $n$: Let $x\in F_e(n), y\in V_1, z\in F_e(n), t\in V_2, x'\in
F_e(n), y'\in W_1, z'\in F_e(n), t'\in W_2, \gamma\in V_1\cap W_2,
\delta\in V_2\cap W_1$ such that $x+y=x'+y'$, $z+t=z'+t'$ and
$(x+y)+(z+t)=\gamma + \delta\in F_e(n+1)\cap F_\tau(1)$.

On has then $x+(y-\gamma)\in F_e(n)+V_1$,$-(z+(t-\delta))\in
F_e(n)+V_2$, et $x+(y-\gamma)=-(z+(t-\delta))$. By application of
proposition~\ref{Fscohomogene} one obtains $y-\gamma\in F_e(n)$ and
$t-\delta\in F_e(n)$. One deduces: $x+y=(x+(y-\gamma))+\gamma\in
F_e(n)+(V_1\cap W_2)$ and also $z+t=(z+(t-\delta))+\delta\in
F_e(n)+(V_2\cap W_1)$. Analogously one proves that
$x'+y'=(x'+(y'-\delta))+\delta\in F_\sigma(n)+(V_2\cap W_1)$ and
$z'+t'=(z'+(t'-\gamma))+\gamma\in F_\sigma(n)+(V_1\cap W_2)$. By a
new application of proposition~\ref{Fscohomogene} (possible by the
fact that $V_1\cap W_2\subset V_1$ and $V_2\cap W_1\subset V_2$) one
obtains that $x+y=x'+y'\in F_e(n)$ and similarly $z+t=z'+t'\in
F_e(n)$. So $(x+y)+(z+t)\in F_e(n)\cap F_\tau(1)$. By induction
hypothesis one has so $(x+y)+(z+t)=0$. \end{proof}

\begin{corr}\label{interfeft}
If $n\ge 1$ then $F_\sigma(n)\cap F(1)=F_\sigma(1)$
\end{corr}
\begin{proof} It is clear that $F_\sigma(1)\subset
F_\sigma(n)\cap F(1)$. For the other inclusion, lets remark first:
$F_\sigma(n)\cap F(1)=F_\sigma(n)\cap (F_\sigma(1)\oplus F_{\bar
\sigma}(1))$. Let $x=a+b\in F_\sigma(n)$ with $a\in F_\sigma(1)$ and
$b\in F_{\bar \sigma}(1)$. $x-a=b\in F_\sigma(n)\cap F_{\bar
\sigma}(1)=\{0\}$. So we have $x\in F_\sigma(1)$. \end{proof}

\begin{prop}
$F_e(n)\cap F_\tau(n)=\{0\}$
\end{prop}
\begin{proof} By induction: It is true for $n=0$. Suppose its
true up to order $n$. Let $x\in F_e(n+1)\cap F_\tau(n+1)$, one
deduces then $\theta(x)\in \theta(F_e(n+1)) \cap
\theta(F_\tau(n+1))\stackrel{\mbox{\scriptsize
prop.~\ref{thetafs}}}{\subset} F_e(n)\cap F_\tau(n)=\{0\}$. From
this we obtain $x\in F(1)\cap F_e(n) \cap F_\tau(n)=(F(1)\cap
F_e(n)) \cap (F(1)\cap F_\tau(n))\stackrel{\mbox{\scriptsize
corr.~\ref{interfeft}}}{=}F_e(1)\cap F_\tau(1)=\{0\}$. \end{proof}

\begin{prop}\label{feftfn}
$\forall n, F_e(n)\oplus F_\tau(n)=F(n)$
\end{prop}

Let's start by proving two lemma:

\begin{lem}
$\forall i,n, F_\sigma(n)\subset V_i+W_{\bar \sigma(i)}$
\end{lem}
\begin{proof} Let's give the proof for $\sigma=\tau$. The proof is
essentially the same in the case $\sigma=e$.

By induction on $n$: For $n=0$ we have $F_\tau(0)=\{0\} \subset
V_i+W_i$. Suppose the result true up to order $n$.
$(F_\tau(n)+V_i)\cap(F_\tau(n)+W_{\bar i})\subset F_\tau(n)+V_i$ and
$(F_\tau(n)+V_{\bar i})\cap(F_\tau(n)+W_i)\subset F_\tau(n)+W_i$. By
summation of the two inclusions one obtains $F_\tau(n+1)\subset
F_\tau(n)+V_i+F_\tau(n)+W_i$. The latter is included in $V_i+W_i$ by
induction hypothesis. \end{proof}

\begin{lem}\label{viwjhomogenefeft}
$\forall n,i,j, V_i+W_j$ is homogeneous with respect to the (direct)
sum $F_e(n)+F_\tau(n)$.
\end{lem}
\begin{proof} Let's make the proof for $i=j=1$, the proof being
similar in the other cases.

The inclusion $(V_1+W_1)\cap F_e(n)+(V_1+W_1)\cap F_\tau(n) \subset
(V_1+W_1) \cap (F_e(n)+F_\tau(n))$ being trivial, let us show the
other inclusion: Let $\alpha\in F_e(n)$, $\beta\in F_\tau(n)$ with
$\alpha+\beta\in V_1+W_1$. By the inclusion $F_\tau(n)\subset
V_1+W_1$ obtained by the preceding lemma one has: $\beta\in
F_\tau(n)\cap (V_1+W_1)$. As $\alpha+\beta\in V_1+W_1$ and $\beta\in
V_1+W_1$ one has $\alpha=(\alpha+\beta)-\beta\in F_e(n)\cap
(V_1+W_1)$.\end{proof}

\begin{proof}{\em proposition~\ref{feftfn}: } By induction on $n$.
For $n=0$ it is evident. Suppose the result proved up to order $n$.

Let us recall that $F(n+1)=\ker \theta^{n+1}$. Let $x\in F(n+1)$. By
induction hypothesis there exists $\alpha \in F_e(n)$, $\beta\in
F_\tau(n)$ such that $\theta(x)=\alpha+\beta$. Set $v_{ij}=(p_i\circ
q_j)(x)$ and $w_{ij}=(q_j\circ p_i)(x)$. Let us remark that
$v_{ij}\in V_i$ and $w_{ij}\in W_j$. Recall that
$w_{ij}=\theta(x)+v_{ij}=\alpha+\beta+v_{ij}$. So one has more
precisely $w_{ij}\in (F_e(n)+F_\tau(n)+V_i)\cap W_j$. As in the
proof of proposition~\ref{kerimtheta} let us remark that
$x=\sum_{i,j} w_{ij}$. If one proves that $w_{ij}\in F_e(n+1) +
F_\tau(n+1)$ the proposition is proved.

As $\alpha+\beta=-v_{ij}+w_{ij}\in (F_e(n)\oplus F_\tau(n))\cap
(V_i+W_j)$ one can apply lemma~\ref{viwjhomogenefeft} in order to
obtain that $\alpha_{ij}\in V_i, \alpha'_{ij}\in W_j$ such that
$\alpha=\alpha_{ij}+\alpha'_{ij}$ and $\beta_{ij}\in V_i,
\beta'_{ij}\in W_j$ such that $\beta=\beta_{ij}+\beta'_{ij}$. One
has: $\alpha'_{ij}=\alpha-\alpha_{ij}\in (F_e(n)+V_i)\cap W_j\subset
F_e(n+1)$ and $\beta'_{ij}=\beta-\beta_{ij}\in (F_\tau(n)+V_i)\cap
W_j\subset F_\tau(n+1)$. On the other hand $W_j\ni
w_{ij}-\alpha'_{ij}-\beta'_{ij}=\alpha_{ij}+(v_{ij}+\beta_{ij})\in
F_e(n)+V_i$ and so $w_{ij}-\alpha'_{ij}-\beta'_{ij}\in
(F_e(n)+V_i)\cap W_j\subset F_e(n+1)$. Finally
$w_{ij}=(w_{ij}-\alpha'_{ij}-\beta'_{ij})+\alpha'_{ij}+\beta'_{ij}\in
F_e(n+1)+F_\tau(n+1)$, and so $x=\sum_{i,j}w_{ij}\in
F_e(n+1)+F_\tau(n+1)$. \end{proof}

\begin{prop}\label{fsv1v2}
\begin{enumerate}
\item $F_\sigma(n)$ is homogeneous with respect to the sum $V_1\oplus V_2$,
equivalently $F_\sigma(n)=(F_\sigma(n)\cap V_1)\oplus
(F_\sigma(n)\cap V_2)$.
\item $F_\sigma(n)$ is homogeneous with respect to the sum $W_1\oplus W_2$, equivalently
$F_\sigma(n)=(F_\sigma(n)\cap W_1)\oplus (F_\sigma(n)\cap W_2)$
\end{enumerate}
\end{prop}
\begin{proof} Let us prove the first point. The proof of the
second is similar. Let $x\in F_\sigma(n)$, $x=y+z$ with $y\in V_1$
and $z\in V_2$. We have then $y=x-z\in V_1 \cap (F_\sigma(n) +
V_2)\subset (F_\sigma(n) + V_1) \cap (F_\sigma(n) +
V_2)\stackrel{Prop.~\ref{Fscohomogene}}{=}F_\sigma(n)$. From this
$y\in F_\sigma(n)\cap V_1$. In the same way $z\in F_\sigma(n)\cap
V_2$. As a conclusion $F_\sigma(n)=(F_\sigma(n)\cap V_1)\oplus
(F_\sigma(n)\cap V_2)$. \end{proof}

\begin{prop}\label{fsvw}
$\forall n, (F_\sigma(n)\cap V_i)\oplus(F_\sigma(n)\cap W_{\bar
\sigma(i)})=F_\sigma(n)$
\end{prop}
\begin{proof} We have from proposition~\ref{Fsfbs1}
$(F_\sigma(n)\cap V_i)\cap(F_\sigma(n)\cap W_{\bar
\sigma(i)})=\{0\}$. Let us write $n_i:=\dim F_\sigma(n)\cap V_i$ and
$m_j :=\dim F_\sigma(n)\cap W_j$. We have from the preceding remark
that $n_i+m_{\bar \sigma(i)}\le \dim F_\sigma(n)$ (*). From
proposition~\ref{fsv1v2} one has $n_1+n_2=\dim F_\sigma(n)$ and
$m_1+m_2=\dim F_\sigma(n)$. By summing the two equalities it is
necessary that (*) is an equality and so $(F_\sigma(n)\cap
V_i)\oplus(F_\sigma(n)\cap W_{\bar \sigma(i)})=F_\sigma(n)$.
\end{proof}

\begin{prop}\label{interplus}
If $A$, $B$ vector subspaces of $E$ are homogeneous with respect to
the sum $\oplus_{i\in I} F_i=E$ then $A+B$ and $A\cap B$ are
homogeneous with respect to the sum $\oplus_{i\in I} F_i$.
\end{prop}
\begin{proof} "$A+B$": Equivalently one has: $\bigoplus_i
(F_i\cap (A+B)) \subset A+B$. Let us show the other inclusion. Let
$x\in A+B=(\bigoplus_i (F_i\cap A))+(\bigoplus_i (F_i\cap B))$. So
one has $x=\sum_i x_i+\sum_i x'_i$ with $x_i\in F_i\cap A$ and$
x'_i\in F_i\cap B$. By writing $x=\sum_i (x_i+x'_i)$ one sees that
$x\in \bigoplus_i (F_i\cap (A+B))$.

"$A\cap B$": Evidently one has: $\bigoplus_i (F_i\cap (A\cap B))
\subset A\cap B$. For the other inclusion let $x\in A\cap
B=(\bigoplus_i (F_i\cap A))\cap (\bigoplus_i (F_i\cap B))$,
$x=\sum_i x_i=\sum_i x'_i$ with $x_i\in F_i\cap A$ and $x'_i\in
F_i\cap B$. By unicity of the decomposition of $x$ with respect to
the direct sum $\bigoplus_i F_i$ it is clear that $\forall i,
x_i=x'_i$ and so that $x\in \bigoplus_i (F_i\cap (A\cap B))$.
\end{proof}


\begin{prop}\label{treillishomogene}
For every element $V$ of the lattice generated by $V_1$, $V_2$,
$W_1$ and $W_2$ one has: $V=(V\cap F_e)\oplus (V\cap F_\tau) \oplus
(V\cap \tilde F)$
\end{prop}
\begin{proof} Due to proposition~\ref{interplus} it is enough to
prove that $V_1$, $V_2$, $W_1$ and $W_2$ are homogeneous with
respect to the sum: $E=F_e\oplus F_\tau\oplus \tilde F$.

Let us prove for this purpose the lemma:

\begin{lem} Let $E$, $E_j$ et $F_i$ be vector spaces.
If $E=E_1\oplus E_2$ with $\forall i, F_i=(F_i\cap
E_1)\oplus(F_i\cap E_2)$, then $E_j\cap \oplus_i F_i=\oplus_i
(E_j\cap F_i)$ for $j=1,2$.
\end{lem}
\begin{proof} $\oplus_i F_i =\oplus_i(F_i \cap (E_1 \oplus E_2))
=\oplus_i((F_i \cap E_1) \oplus (F_i \cap E_2))=\oplus_i (F_i \cap
E_1)\oplus \oplus_i (F_i \cap E_2)$. But $\oplus_i (F_i \cap E_j)
\subset(\oplus_i F_i) \cap E_j$. As
 $((\oplus_i F_i) \cap E_1) \oplus ((\oplus_i F_i) \cap E_2) \subset \oplus F_i$,
 the inclusions in this proof are necessarily equalities.
So $\oplus_i (F_i\cap E_j)=(\oplus_i F_i)\cap E_j$. \end{proof}
\medskip

{\em end of proof of proposition~\ref{treillishomogene}:} By
applying the lemma for $\forall i, E_i=V_i$ (respectively $\forall
i, E_i=W_i$) proposition~\ref{fsv1v2} and
proposition~\ref{tildefv1v2} show that $V_1$, $V_2$, $W_1$ and $W_2$
are homogeneous with respect to the sum decomposition: $E=F_e\oplus
F_\tau\oplus \tilde F$. \end{proof}

\subsection{Reflexive case}\label{orth}\index{forme!réflexive}

Suppose that $E$, $V_1$, $V_2$ are finite-dimensional vector-spaces
such that $E=V_1 \oplus V_2$ and suppose that $E$ carries a non
degenerate reflexive form $a$. We have seen that $(E,V_1, V_2,
V_1^\perp, V_2^\perp)$ is a decomposition of $E$ into two direct
sums. Suppose $F(n)$, $F$, $F_\sigma(n)$, $F_\sigma$, $\tilde F(n)$
and $\tilde F$ defiend as before.

Let's prove the following proposition:
\begin{prop}\label{ffforth}
$$F=F_e \oplus^\perp F_\tau \oplus^\perp \tilde F$$
\end{prop}
\begin{proof} For $\sigma\in {\mathcal S}_2$ let $\tilde
F_\sigma(0):= E$ and $\tilde F_\sigma(n+1):= \bigcap_i ((\tilde
F_\sigma(n)\cap V_i)+(\tilde F_\sigma(n)\cap W_{\sigma(i)}))$. The
sequence $\tilde F_\sigma(n)$ is decreasing and so stationary in
finite dimensions. Note $\tilde F_\sigma:=\bigcap_n \tilde
F_\sigma(n)$.

By induction it is easy to see that\footnote{By using again the fact
that $(A+B)^\perp=A^\perp\cap B^\perp$ and $(A\cap B)^\perp=A^\perp+
B^\perp$ for $A$ and $B$ vector subspaces of $E$} $\forall n, \forall
\sigma\in{\mathcal S}_2, F_\sigma(n)^\perp = \tilde F_\sigma(n)$.

By writing the definition of $\tilde F(n)$ and $\tilde F_\sigma(n)$
it is easy to see by induction that $\forall n, \forall \sigma,
\tilde F(n) \subset \tilde F_\sigma(n)$, from which we obtain
$\forall \sigma, \tilde F \subset \tilde F_\sigma$.

In order to finish the proof lets show the following lemma:

\begin{lem}
For $\sigma\in {\mathcal S}_2$ we have: $\forall n, F_{\bar \sigma}
\subset \tilde F_{\sigma}(n)$
\end{lem}
\begin{proof} By induction on $n$: It is clear for $n=0$. For
$n+1$ we have: $\tilde F_{\sigma}(n+1)=\bigcap_i ((\tilde
F_{\sigma}(n)\cap V_i)+(\tilde F_{\sigma}(n)\cap W_{\sigma(i)}))
\subset \bigcap_i ((F_{\bar \sigma}\cap V_i)+(F_{\bar \sigma}\cap
W_{\sigma(i)}))$ by induction hypothesis. The latter expression is
equal to $F_{\bar \sigma}$ by proposition~\ref{fsvw}. \end{proof}

{\em end of the proof of proposition~\ref{ffforth}:} As
$\dim(F_\sigma)+\dim(\tilde F_\sigma)=\dim(E)=\dim(F_\sigma) +
\dim(F_{\bar \sigma}) + \dim(\tilde F)$ (because $F_\sigma$ and
$\tilde F_\sigma$ are orthogonal, respectively by
proposition~\ref{treillishomogene}) we have $\dim(\tilde F_\sigma)=
\dim(F_{\bar \sigma}) + \dim(\tilde F)$. By the inclusion $F_{\bar
\sigma} \oplus \tilde F \subset \tilde F_{\sigma}$ we must have
$F_{\bar \sigma} \oplus \tilde F = \tilde F_{\sigma}$. \end{proof}

\subsection{Sublattice "with 5 direct sums"} It is known that the
lattice generated by the three vector subspaces of $E$: $U,V,W$ such
that $E =U \oplus W = V \oplus W$ has the following structure:

\begin{center}
\includegraphics[origin=c]{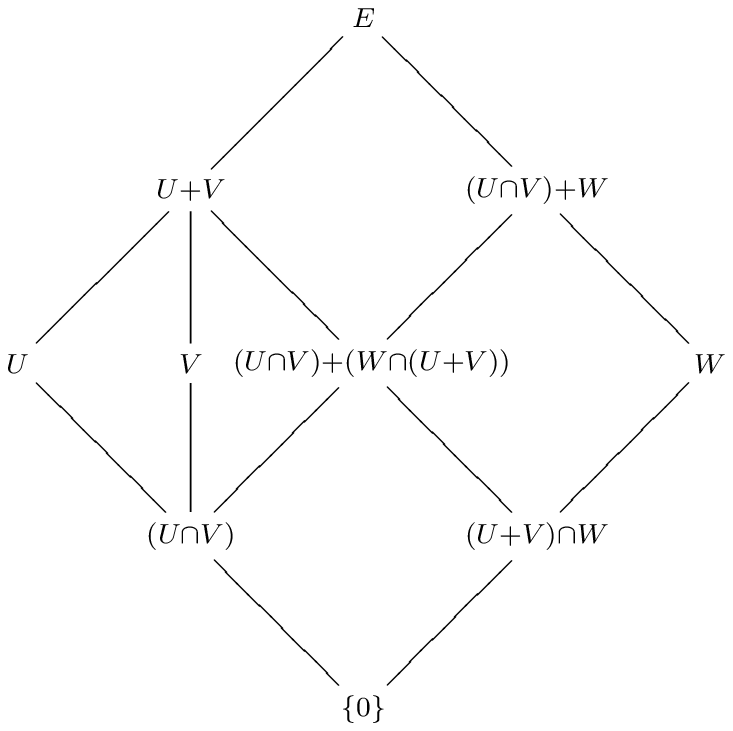}
\end{center}

The construction applies to the lattice $T$ generated by the $4$
subspaces of $E$, $V_1$, $V_2$, $W_1$, $W_2$ such that $E=V_1\oplus
V_2=W_1 \oplus W_2 = V_1\oplus W_2 =W_1 \oplus V_2$, in the
following way:

We can choose for $(U,V,W)$ the triple $(V_1,W_1,V_2)$ or
$(V_1,W_1,W_2)$. Note that then in the first case: $T_1:=(V_1 \cap
W_1) + (V_2 \cap (V_1 + W_1))$ and in the second: $U_1:=(V_1 \cap
W_1) + (W_2 \cap (V_1 + W_1))$.

The interval $[V_1 \cap W_1, V_1 + W_1]$ is a sub-lattice $T'$ of
$T$ which contains is particular the elements $V'_1:=V_1/(V_1 \cap
W_1)$, $W'_1:=W_1/(V_1 \cap W_1)$, $T'_1:=T_1/(V_1 \cap W_1)$ et
$U'_1:=V_1/(V_1 \cap W_1)$ verifying:

$$V'_1 \oplus W'_1 = V'_1 \oplus T'_1 = V'_1 \oplus U'_1 = W' \oplus
T'_1 = W'_1 \oplus U'_1$$

On the other hand it is possible that $T'_1 \cap U'_1 \neq \{0\}$
(as well as $T'_1 + U'_1 \neq (V_1 + W_1)/(V_1 \cap W_1)$).

Note is particular that $T'$ contains two sublattices of type $M_3$:
The one constructed on the elements $\{\{0\}, E ,V'_1, W'_1, T'_1\}$
and the one given by the elements $\{\{0\}, E ,V'_1, W'_1, U'_1\}$.

The data of $T'_1$ is equivalent to the data of an isomorphism $i$
of $V'_1$ onto $W'_1$, and the data of $U'_1$ of a second
isomorphism $j$ of $V'_1$ onto $W'_1$. the conjugation class in
$Gl(V'_1)$ of $j^{-1}\circ i$ is then an invariant of the lattice.
We can compare this result to the operators that Gelfand and
Ponomarev used in their paper~\cite{gp}.

\subsection{Example}\label{exemple}

In this paragraph we are going to study the structure of the lattice
generated by four finite-dimensional vector spaces $V_1, V_2, W_1,
W_2$ such that $E=V_1\oplus V_2=W_1 \oplus W_2 = V_1\oplus W_2 =W_1
\oplus V_2$ supposing that $\theta_{\mathcal V}^2 =0$ for ${\mathcal
V}=(E, V_1, V_2, W_1, W_2)$.

\begin{lem}
On a: $(V_1 + W_1)\cap V_2 = (V_1 + W_1)\cap W_2 \subset V_2 \cap
W_2$ et $(V_2 + W_2)\cap V_1 = (V_2 + W_2)\cap W_1 \subset V_1 \cap
W_1$
\end{lem}
\begin{proof} It is clear that $(V_1 + W_1)\cap V_2 \subset (V_1
+ W_1)\cap (V_2+W_2) =\im \theta \subset \ker \theta=(V_1 \cap
W_1)\oplus (V_2 \cap W_2)$. From which one can see that $(V_1 +
W_1)\cap V_2 \subset \ker \theta \cap V_2 = V_2 \cap W_2$. So $(V_1
+ W_1)\cap V_2 \subset (V_1 + W_1)\cap W_2$ and similarly $(V_1 +
W_1)\cap W_2 \subset (V_1 + W_1)\cap V_2$, which proves the first
assertion. The proof of the second one is similar.\end{proof}

Note $X_0=\{0\}, X_1=(V_2 + W_2)\cap V_1, X_2=V_1 \cap W_1, X_3=V_1$
et $Y_0=\{0\}, Y_1=(V_1 + W_1)\cap V_2, Y_2=V_2 \cap W_2, Y_3=V_2$.

As $X_0 \subset X_1 \subset X_2 \subset X_3=V_1$ and $Y_0 \subset
Y_1 \subset Y_2 \subset Y_3=V_2$ and $V_1 \cap V_2 =\{0\}$, it is
easy to see that the lattice ${\mathcal T}_0$ generated by the $X_i$
and the $Y_j$ for $i,j=0, 1, 2, 3$ is precisely the set $\{X_i
\oplus Y_j \; \vrule \; i,j=0, 1, 2, 3\}$, ordered by inclusion.

It is easy to verify that $X_i \oplus Y_j= (X_i\oplus V_2) \cap (V_1
\oplus Y_j)$ and so the lattice ${\mathcal T}_0$ can be written as
well:

Note $X'_i=X_i\oplus V_2$ et $Y'_j=V_1\oplus Y_j$. We have then:
$X'_0=V_2$, $X'_1=((V_2 + W_2)\cap V_1) + V_2$, $X'_2=(V_1 \cap
W_1)+ V_2$, and $X'_3=V_1 + V_2$.

Let's verify: $X'_1=V_2+W_2$. It is clear that $X'_2 \subset
V_2+W_2$ inversely if $x\in V_2$ and $y\in W_2$, $x+y$ can be
written uniquely $a+b$ with $a\in V_1$ and $b\in V_2$, so
$a=((x-b)+y)+b \in (V_2 + W_2)\cap V_1$ and so $a+b \in ((V_2 +
W_2)\cap V_1)+V_2$. So $X'_1=V_2+W_2$.

We have as well: $Y'_0=V_1$, $Y'_1=V_1 + W_1$, $Y'_2=(V_2 \cap W_2)+
V_1$, et $Y'_3=V_1 + V_2$.

The underlying set of the lattice ${\mathcal T}_0$ is so: $\{X'_i
\cap Y'_j \; \vrule \; i,j=0, 1, 2, 3\}$.

We are going to prove that ${\mathcal T} = {\mathcal T}_0 \cup
\{W_1, W_2\}$ is a lattice. Let us verify that $\mathcal T$ is
stable by intersection and sum.

Verify that $(X_i\oplus Y_j)+W_1 \in {\mathcal T}$: If $j=0$ and
$i=0, 1,2$ it is clear that $(X_i\oplus Y_j)+W_1= W_1 \in {\mathcal
T}$. If $j=0$ and $i=3$, $(X_i\oplus Y_j)+W_1=V_1 + W_1 \in
{\mathcal T}$. If $j\ge 1$, $(X_i\oplus Y_j)+W_1 = (Y_1 + W_1 + X_i
+ Y_j)$. By a similar argument to the one which allowed us to have
before: $((V_2 + W_2)\cap V_1) + V_2=V_2 + W_2$, one can prove $Y_1
+W_1=((V_1 + W_1)\cap V_2) + W_1=V_1 + W_1 \in {\mathcal T}_0$, and
so $Y_1 + W_1 + X_i + Y_j\in{\mathcal T}_0 \subset {\mathcal T}$.

By using the second representation of ${\mathcal T}_0$ we can show
for every $i$ and $j$, $(X'_i\cap Y'_j) \cap W_1 \in {\mathcal T}$.
The only delicate point is to verify that $((V_1 \cap W_1)+ V_2)
\cap W_1=V_1 \cap W_1$. Let's do it: It is clear that $V_1 \cap W_1
\subset ((V_1 \cap W_1)+ V_2) \cap W_1$, inversely let $x\in V_1
\cap W_1$, $y\in V_2$ and $z\in W_1$ such that $x+y=z$. We have
then: $y=z-x \in V_2 \cap W_1 =\{0\}$, and so $z=x \in V_1 \cap
W_1$.

In conclusion we can state:

\begin{theo}

The structure of the lattice generated by the four
finite-dimensional vector spaces $V_1, V_2, W_1, W_2$ such that
$E=V_1\oplus V_2=W_1 \oplus W_2 = V_1\oplus W_2 =W_1 \oplus V_2$ and
supposing that $\theta_{\mathcal V}^2 =0$ for ${\mathcal V}=(E, V_1,
V_2, W_1, W_2)$ is given by the following diagram:

\begin{center}
\includegraphics[origin=c]{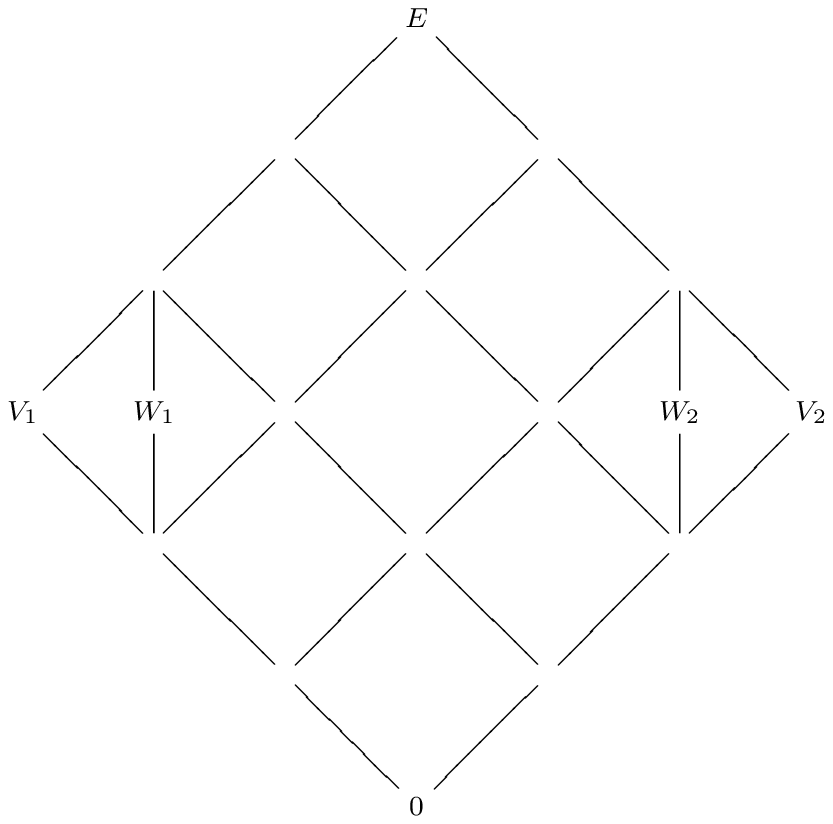}
\end{center}

\end{theo}

\section{Application to representation theory}
\subsection{Preliminaries}
\subsubsection{General case}
We will note $\g{g}\g{l}(V_1,V_2,W_1,W_2)$ the set of $a\in
\g{g}\g{l}(E)$ such that $a V_i \subset V_i$ et $a W_j \subset W_j$.
It is easy to see that $\g{g}\g{l}(V_1,V_2,W_1,W_2)$ is a sub
Lie-algebra of $\g{g}\g{l}(E)$. Let $\g{g}$ be a sub Lie-algebra of
$\g{g}\g{l}(V_1,V_2,W_1,W_2)$. We have for all $A, B$
vector subspaces of $E$ such that $\g{g} A \subset A$ et $\g{g} B
\subset B$: $\g{g} (A+B) \subset (A+B)$ et $\g{g} (A\cap B) \subset
(A\cap B)$. So we have, as $\g{g}$ leaves invariant $V_1,V_2,W_1$ et
$W_2$, $\g{g}$ leaves invariant every element of the lattice
generated from $V_1,V_2,W_1$ et $W_2$ by intersection and sum.

It is easy to see that the projections $p_{V_i}^{V_j}$ and
$p_{W_i}^{W_j}$ commute to the action of $\g{g}$: $\forall a\in
\g{g}, a p_{V_i}^{V_j}= p_{V_i}^{V_j} a$ and $a p_{W_i}^{W_j}=
p_{W_i}^{W_j} a$. So every element of the associative unitary
algebra $A$ generated by the $p_{V_i}^{V_j}$ and the $p_{W_i}^{W_j}$
commutes to every $a\in \g{g}$. As an example
$\theta=[p_{W_1}^{W_2},p_{V_1}^{V_2}]$ commutes to every $a\in
\g{g}$.

\begin{lem}
The data of two supplementary vector-spaces $V_1$ et $V_2$ stable for
the action of a linear Lie algebra $\g{g}$ is equivalent to the data
of an endomorphism $L$ commuting with the action of $\g{g}$,
verifying $L^2=I$. $V_1$ et $V_2$ are then the proper subspaces of
$L$ associated to the eigenvalues $1$ and $-1$.
\end{lem}
\begin{proof} In fact it is easy to see that the endomorphism
$L=p_{V_1}^{V_2}-p_{V_2}^{V_1}$ is of square identity and commutes
to the action of $\g{g}$. Inversely if an endomorphism $L$ is such
that $L^2=I$ and commutes to the action of $\g{g}$, it admits the
proper values $1$ and/or $-1$. The corresponding eigenspaces are
supplementary and stable for the action of $\g{g}$.\end{proof}

\subsubsection{Reflexive case}

We recall that in the reflexive case we suppose that there exists a
non degenerate reflexive form $\langle \cdot,\cdot\rangle$ such that
$\forall a\in \g{g}$, $\forall x,y\in E$, we have: $\langle
ax,y\rangle +\langle x,ay\rangle =0$.

Recall as well that if $V$ is a subspace of $E$ which is
$\g{g}$-invariant then $V^\perp$ is invariant as well. We suppose
here that $W_1=V_1^\perp$, $W_2=V_2^\perp$. These two spaces are
supplementary and invariant.

\begin{lem}
Let $L^*$ be the adjoint with respect to a reflexive form of the
endomorphism $L=p_{V_1}^{V_2}-p_{V_2}^{V_1}$ which commutes to the
action of $\g{g}$ and is such that $L^2=I$. Then $L^*$ is of square
identity, commutes to the action of $\g{g}$ and one has:
$$L^*=p_{V_2^\perp}^{V_1^\perp}-p_{V_1^\perp}^{V_2^\perp}$$
\end{lem}
\begin{proof} Let's note
$L':=p_{V_2^\perp}^{V_1^\perp}-p_{V_1^\perp}^{V_2^\perp}$ and let us
show that $\forall v,w\in E, \langle Lv,w\rangle =\langle
v,L'w\rangle$,

We write for $x\in V_1,x'\in V_2,y\in V_1^\perp,y'\in V_2^\perp$,
\begin{eqnarray*}
\langle L(x+x'),y+y'\rangle & = & \langle x-x',y+y'\rangle\\
& = & \langle x,y'\rangle-\langle x',y\rangle\\
& = & \langle x+x',-y+y'\rangle\\
& = & \langle x+x',L'(y+y')\rangle
\end{eqnarray*}
As a consequence $L^*=L'$.\end{proof}

\medskip

Let's remark that $L=-L^*$ for $L=p_{V_1}^{V_2}-p_{V_2}^{V_1}$ is
equivalent to have $V_1=V_1^\perp$ and $V_2=V_2^\perp$. It is the
same to impose $\langle Lx, Ly \rangle=-\langle x, y \rangle$ for
$x,y\in E$ {\em i.e.} $L$ is anti-hermitian with respect to the
reflexive form.

The data of $L\in End(E)$ such that $L^2=Id$ and of a reflexive form
for which $L$ is anti-hermitian is also called a para-Kähler
structure.

\medskip
We recall that the reflexive representation $\g g \subset \g{gl}(E)$
is called {\em weakly irreducible} if any invariant subspace
$V\subset E$ is either $\{0\}$, $E$, or is degenerate {\em i.e.}
$V\cap V^\perp\neq \{0\}$.

As we saw in paragraph~\ref{orth}, if $W_1=V_1^\perp$ and
$W_2=V_2^\perp$, we have in the weakly irreducible case and if $V_1$
et $V_2$ are different from $\{0\}$ necessarily $E=F_e$. In fact if
two of the three spaces $F_e$, $F_\tau$, $\tilde F$ are non trivial
then $E$ is not weakly irreducible. The more in the case $E=F_\tau$
or $E=\tilde F$, the fact that $E=V_1 \oplus V_1^\perp$ would imply
that if $E$ is non trivial, $E$ is not weakly irreducible.

\begin{prop}\label{identification_au_dual}
In the case the representation $E=V_1\oplus V_2$ is weakly
irreducible and if $V_1$ and $V_2$ are different from $\{0\}$, $V_2$
identifies (as a representation) to the dual $V_1^*$ of $V_1$.
\end{prop}
\begin{proof} It identifies by the map
$$\begin{array}{l}
V_2 \to V_1^*\\
v' \mapsto (w \mapsto \langle v',w\rangle )
\end{array}
$$
which is injective by the fact that $V_1 \cap V_2^\perp =\{0\}$ and
surjective for dimension reasons. In fact we have $V_1 \oplus
V_2^\perp = V_1 \oplus V_2$ implies that $dim(V_2)=dim(V_2^\perp)$.
From this we obtain $dim(V_2)=\frac{1}{2}dim(E)$ and similarly
$dim(V_1)=dim(E)-\frac{1}{2}dim(E)=\frac{1}{2}dim(E)$. As
$dim(V_1^*)=dim(V_1)$, we have: $dim(V_1^*)=dim(V_2)$. \end{proof}

\subsection{Main result}
The following result could be formulated thanks to a suggestion of
Martin Olbrich. He communicated to us a direct proof of the
result~\ref{Olbrich}, which we had established for pseudo-riemannian
holonomy algebras only.

\begin{theo}\label{deux_isotropes}
If $E$ is a representation admitting two decompositions into
supplementary sub-representations $E=V_1\oplus V_2=W_1 \oplus W_2$,
then, noting $E_{(L,\lambda)}$ the generalized eigenspace associated
to the eigenvalue $\lambda$ for the operator $L$, we have:
\begin{enumerate}[(i)]
\item $F_e=E_{(L,-1)} \oplus E_{(L,1)}$ as a representation for the invariant operator $L=p_{V_1}^{V_2}-p_{W_2}^{W_1}$.
The more we have $V_1\cap W_1 \subset E_{(L,1)}$ and $V_2\cap W_2
\subset E_{(L,-1)}$
\item $F_\tau=E_{(L',-1)} \oplus E_{(L',1)}$ as a representation for the invariant operator $L'=p_{V_1}^{V_2}-p_{W_1}^{W_2}$.
The more we have $V_1\cap W_2 \subset E_{(L',1)}$ et $V_2\cap W_1
\subset E_{(L',-1)}$
\end{enumerate}
\medskip

When $E$ is in addition reflexive and $W_j=V_j^\perp$, then
\begin{enumerate}[(i)]
\item
$L$ is anti-self-adjoint with respect to the reflexive form,
$E_{(L,-1)}$ and $E_{(L,1)}$ are totally isotropic and their direct
sum is non degenerate.
\item
$L'$ is self-adjoint with respect to the reflexive form,
$E_{(L',-1)}$ and $E_{(L',1)}$ are non degenerate and orthogonal.
\end{enumerate}
\end{theo}
\begin{proof} It follows from the fact that the spaces $F_e$, $F_\tau$ and
$\tilde F$ are homogeneous, that $L=p_{V_1}^{V_2}-p_{W_2}^{W_1}$
(and similarly $L'=p_{V_1}^{V_2}-p_{W_1}^{W_2}$) is an endomorphism
of each of these spaces.

For $\sigma=e$ or $\tau$ note
$P_\sigma(X)=\Pi_{\lambda\in\Lambda_\sigma}
P_{\sigma,\lambda}^{n_\lambda}(X)$ the minimal polynomial of $L$
restricted to $F_\sigma$ and similarly $\tilde
P(X)=\Pi_{\lambda\in\tilde \Lambda} \tilde
P_{\lambda}^{n_\lambda}(X)$ the minimal polynomial of $L$ restricted
to $\tilde F$.

$F_\sigma$ decomposes into the generalized eigenspaces
${F_\sigma}_{(L,\lambda)}:=\ker(P_{\sigma,\lambda}^{n_\lambda}(L\vrule_{F_\sigma})$.
and $\tilde F$ decomposes into the generalized eigenspaces ${\tilde
F}_{(L,\lambda)}:=\ker(\tilde
P_{\lambda}^{n_\lambda}(L\vrule_{\tilde F})$.

Let's make the convention that $P_{\sigma,\lambda}(X)=X+ \lambda$
and $\tilde P_{\lambda}(X)=X+\lambda$ for $\lambda=0,-1,1$.

It is immediate that: $V_1\cap W_1\subset F_{e,(L,1)}$ and $V_2\cap
W_2\subset F_{e,(L,-1)}$.

It is easy to verify from the definitions that $\theta L =- L
\theta$. On deduces that $\theta$ maps ${F_\sigma}_{(L,\lambda)}$
into ${F_\sigma}_{(L,\lambda')}$ with $P_{\sigma,\lambda'}(X)=\pm
P_{\sigma,\lambda}(-X)$.

Similarly $\theta$ maps ${\tilde F}_{(L,\lambda)}$ into ${\tilde
F}_{(L,\lambda')}$ with $\tilde P_{\lambda'}(X)=\pm \tilde
P_{\lambda}(-X)$.

Let $x\in {F_e}_{(L,\lambda)}$ and let $n$ be the smallest integer
such that $\theta^{n+1}(x)=0$, which exists from the fact that
$\theta$ is nilpotent on $F_e$. $\theta^n(x)\in \ker(\theta)\subset
V_1\cap W_1 \oplus V_2\cap W_2\subset {F_e}_{(L,1)}\oplus
{F_e}_{(L,-1)}$. As a consequence $\lambda=\pm 1$ and
$F_e={F_e}_{(L,1)}\oplus {F_e}_{(L,-1)}$

An analogous argument gives $F_\tau={F_\tau}_{(L,0)}$.

Finally let us show that $\lambda = 0, 1, -1\not\in \tilde \Lambda$.
Suppose the contrary. It exists then an eigenvector $x$ in $\tilde
F$ associated to the eigenvalue $\lambda$.
$L(x)=p_{V_1}^{V_2}(x)-p_{W_2}^{W_1}(x)=\lambda x$ implies in the
three cases a contradiction with proposition~\ref{tildefv1v2}.

It follows that $F_e=E_{(L,-1)}\oplus E_{(L,1)}$, as
$E_{(L,\lambda)}={F_e}_{(L,\lambda)}\oplus {F_\tau}_{(L,\lambda)}
\oplus {\tilde F}_{(L,\lambda)}$.

The same arguments show {\em mutatis mutandis} that $F_\tau=
{F_\tau}_{(L',1)}\oplus {F_\tau}_{(L',-1)}$, $V_1 \cap W_2 \subset
{F_\tau}_{(L',1)}$, $V_2 \cap W_1 \subset {F_\tau}_{(L',-1)}$, and
$F_e={F_e}_{(L',0)}$.

It follows similarly $F_\tau=E_{(L',-1)}\oplus E_{(L',1)}$.

The generalized eigenspaces appearing in the proof are invariant by
the fact that for any polynomial $Q$, $Q(L)$ commutes to the action
of the representation and so $\ker Q(L)$ (and also $\im Q(L)$) is
invariant.

In the reflexive case we have: $L=-L^*$. As a consequence
$E_{(L,-1)}$ is orthogonal to any $E_{(L,\lambda)}$ for $\lambda\neq
1$ and $E_{(L,1)}$ is orthogonal to any $E_{(L,\lambda)}$ for
$\lambda\neq -1$. This follows from the relation
$$\langle P_\lambda(L)^{n_\lambda}\cdot, \cdot\rangle=\langle \cdot, P_\lambda(L^*)^{n_\lambda}\cdot\rangle
=\langle \cdot, P_\lambda(-L)^{n_\lambda}\cdot\rangle,$$ and from
the fact that $P_\lambda(L)^{n_\lambda}$ is an isomorphism of
$E_{(L,\mu)}$ for $\mu\neq\lambda$.(kernel lemma)

So $E_{(L,-1)}$, and $E_{(L,-1)}$ are totally isotropic,
$E_{(L,-1)}\oplus E_{(L,1)}$ is orthogonal to all other generalized
eigenspaces and non degenerate.

One obtains similarly that $L'={L'}^*$. $E_{(L',\lambda)}$ is
orthogonal to any $E_{(L',\mu)}$ for $\mu\neq \lambda$. In
particular $E_{(L',\lambda)}$ is non degenerate and $E_{(L',-1)}$ is
orthogonal to $E_{(L',1)}$. \end{proof}

Let us remark that in the weakly irreducible case, the existence of
a decomposition of $E$ into two a direct sum of two degenerate
sub-representations implies that $E=F_e$.

\begin{theo}\label{Olbrich}
If $E$ is a weakly irreducible representation preserving the non
degenerate reflexive form $\langle \cdot, \cdot\rangle$ and
admitting a decomposition into a direct sum of degenerate
sub-representations $E=V_1 \oplus V_2$, then $E=E_{(L,1)}\oplus
E_{(L,-1)}$ with $L:=p-p^*$. We have: $V_1\cap V_1^\perp \subset
E_{(L,1)}$ and $V_2\cap V_2^\perp \subset E_{(L,-1)}$. In addition
$E_{(L,1)}$ et $E_{(L,-1)}$ are totally isotropic and their sum is
non degenerate.
\end{theo}

\begin{prop}\label{identification_au_dual_2}
If $E=E_1 \oplus E_2$ is a representation preserving the non
degenerate reflexive form $\langle \cdot, \cdot\rangle$, and $E_1$
and $E_2$ are totally isotropic, then $E_2$ identifies to $E_1^*$.
\end{prop}
\begin{proof} As in proposition~\ref{identification_au_dual} the map
$$\begin{array}{l}
E_2 \to E_1^*\\
v' \mapsto (w \mapsto \langle v',w\rangle )
\end{array}
$$
which is injective because $E_2 \cap E_1^\perp=\{0\}$ and surjective
for dimension reasons. \end{proof}

\begin{lem}
If the representation $E$ admits three sub-representation $F_1$,
$F_2$ and $F_3$, such that $E=F_1\oplus F_2=F_2\oplus F_3=F_1\oplus
F_3$, then $E=F_1 \otimes \corps^2$ where $\corps^2$ is the trivial
representation.
\end{lem}
\begin{proof}
Let's note $p$ the projection on $F_1$ parallely to $F_2$ restricted
to $F_3$. $p$ is an isomorphism of $F_3$ onto $F_1$ and commutes
with the action of the representation. As a consequence $E=F_1\oplus
F_1=F_1 \otimes \corps^2$.
\end{proof}

\begin{prop}
If $E$ is a representation admitting two decompositions into
supplementary sub-representations $E=V_1\oplus V_2=W_1 \oplus W_2$,
$\tilde F$ identifies to $V \otimes \corps^2$ where $V=\tilde F \cap
V_1$ and $\corps^2$ is the trivial representation.
\end{prop}
\begin{proof} In fact we have $\tilde F=\tilde F \cap V_1 \oplus \tilde F
\cap V_2=\tilde F \cap V_1 \oplus \tilde F \cap W_1=\tilde F \cap
V_2 \oplus \tilde F \cap W_1$. We are in the situation described by
the preceding lemma.

\end{proof}

To summarize we have:

\begin{theo}\label{ts}
If $E$ is a representation preserving the non degenerate reflexive
form $\langle \cdot, \cdot\rangle$ and the direct sum decomposition
$E=V_1 \oplus V_2$, then
\begin{enumerate}[(i)]
\item $E=F_e \oplus^\perp F_\tau \oplus^\perp \tilde F$,
\item $F_e=F_e^+\oplus (F_e^+)^*$ for a totally isotropic representation $F_e^+$,
\item $F_\tau=F_\tau^+\oplus^\perp F_\tau^-$ for a non degenerate representation $F_\tau^+$,
\item $\tilde F=\tilde F_0\otimes \corps^2$ for a non degenerate representation $\tilde F_0$ and $\corps^2$
being the trivial representation.
\end{enumerate}
\end{theo}

\section{Application to holonomy}

A particular case of the preceding is when $\g{g}$ is a holonomy
algebra. We call {\em formal curvature tensor} an element $R$ of
$(E^* \wedge E^*)\otimes E^* \otimes E$ such that for all $x,y,z \in
E$ we have: $R(x,y)z + R(y,z)x +R(z,x)y =0$ (first Bianchi
identity). We will suppose the that there is a finite set of formal
curvature tensors $\{R_1, R_2, \ldots , R_m\}$ such that $\g{g}$ is
the linear Lie algebra generated by the $R_i(x,y) \in End(E)$ for
$i=1 \ldots m$ and $x,y\in E$. We will call such an algebra {\em
Berger algebra}. For a holonomy algebra this situation is given by
the Ambrose-Singer theorem which relates the curvature tensor of a
connected manifold equipped with a torsion-free connection to its
holonomy algebra in a point of the manifold. In the following we
will write $R$ one of the formal curvature tensors $R_1, R_2,
\ldots, R_m$.

\begin{defi}
If $R$ is a formal curvature tensor and $\g g\subset \g{gl}(E)$ a
Berger algebra, we say that $R$ matches $\g g$, if $\forall x,y\in
E, R(x,y)\in \g g$.
\end{defi}

\subsection{General case}

\begin{lem}
If $\g g\subset \g{gl}(E)$ is a Berger algebra admitting the
invariant spaces $F_1, F_2, \ldots, F_r$ with $E=F_1\oplus F_2
\oplus \cdots \oplus F_r$, and if $R$ is a formal curvature tensor
which matches $\g g$, then $\forall i,j,k, k\not\in\{i,j\}
\Rightarrow \forall x\in F_i, y\in F_j, z\in F_k, R(x,y)z=0$.
\end{lem}
\begin{proof}
Suppose $x,y,z$ as in the statement. Then by the identity
$$R(x,y)z+R(y,z)x+R(z,x)y=0$$ and by the fact that $R(y,z)x\in F_i$, $R(z,x)y\in
F_j$ and $R(x,y)z\in F_k$ it is clear from $(F_i+F_j)\cap F_k=\{0\}$
that $R(x,y)z=0$.
\end{proof}

\begin{defi}
We will say that the representation $\g g\subset \g{gl}(E)$
admitting the invariant spaces $F_i$ with $E=F_1\oplus F_2 \oplus
\cdots \oplus F_r$ {\em decomposes into an exterior product along
the decomposition $E=F_1\oplus F_2 \oplus \cdots \oplus F_r$} if for
any $a\in \g g$, $\forall i, a\vrule_{F_i}\in \g g$.
\end{defi}

\begin{prop}
If $\g g\subset \g{gl}(E)$ is a Berger algebra and preserves $V_1$,
$V_2$, $W_1$ and $W_2$ such that $E=V_1\oplus V_2=W_1\oplus W_2$
then $E$ decomposes into an exterior product along the decomposition
$F \oplus \tilde F$. If in addition $\g g$ preserves the reflexive
form $\langle\cdot,\cdot\rangle$ and if $W_1=V_1^\perp$ and
$W_2=V_2^\perp$, then $E$ decomposes into an exterior product along
the decomposition $F_e\oplus F_\tau\oplus \tilde F$.
\end{prop}
\begin{proof} For the first affirmation, this results from the preceding lemma
and from the fact that $\tilde F$ is of type $\tilde
F_0\otimes\reels^2 \simeq \tilde F_0\oplus \tilde F_0$. In the
reflexive case $F_e$ is of type $F_e^+\oplus(F_e^+)^*$ from which by
a similar argument one can deduce the second affirmation.
\end{proof}

\subsection{Metric case} In the metric case the invariant non
degenerate reflexive form $\langle \cdot,\cdot\rangle$ is supposed
to be bilinear symmetric and  $\corps=\reels$.

It is well known that from the invariance of $\langle
\cdot,\cdot\rangle$, the first Bianchi identity and from the
antisymmetry in the two first arguments of $R$, one can deduce
$$\forall x,y,z,t\in E, \langle R(x,y)z,t\rangle =\langle
R(z,t)x,y\rangle,(*)$$ for any formal curvature tensor $R$ matching
the algebra.

\begin{lem}
If the algebra $\g g$ is Berger, preserves two supplementary spaces
$V_1$ et $V_2$ and a non degenerate symmetric bilinear form $\langle
\cdot, \cdot \rangle$, and if for ${\mathcal V}=(E, V_1, V_2,
V_1^\perp, V_2^\perp)$ $E=F_e$, then one has for any formal
curvature tensor $R$ matching $\g g$ and $x,y\in V_1$, $R(x,y)=0$
and for $x',y'\in V_2$, $R(x',y')=0$.
\end{lem}
\begin{proof} From the first Bianchi identity one has $\forall z'\in V_2,
R(x,y)z'+ R(y,z')x+ R(z',x)y=0$. We have: $R(x,y)z'\in V_2$,
$R(y,z')x \in V_1$ and $R(z',x)y\in V_1$ by invariance of $V_1$ and
$V_2$ under the action of $R(x,y)\in\g{g}$ (respectively
$R(y,z')\in\g{g}$, $R(z',x)\in\g{g}$. As $V_1$ and $V_2$ form a
direct sum, one has: $R(x,y)z'=0$.

Let's show us further $\forall z\in V_1, R(x,y)z=0$. Let $t'\in
V_2$. $\langle R(x,y)z,t'\rangle =-\langle z,R(x,y)t'\rangle =0$, by
the preceding argument. So from $R(x,y)z\in V_1$, it is clear that
$R(x,y)z\in V_1\cap V_2^\perp=\{0\}$ (in $F_e$).

As a conclusion for $x,y\in V_1, R(x,y)=0$. Similarly for $x',y'\in
V_2, R(x',y')=0$.\end{proof}

\begin{theo}
If the algebra $\g g\subset\g{gl}(E)$ is Berger, preserves the the
two supplementary spaces $V_1$ and $V_2$ and a non degenerate
symmetric bilinear form $\langle \cdot, \cdot \rangle$, for
${\mathcal V}=(E, V_1, V_2, V_1^\perp, V_2^\perp)$, one has: $\g{g}E
\subset \ker \theta_{\mathcal V}$ and $\g{g}\im \theta_{\mathcal V}=
\{0\}$.
\end{theo}
\begin{proof} By theorem~\ref{ts} one has the decomposition into sub-representations $E=(F_e^+\oplus (F_e^+)^*)\oplus^\perp F_\tau^+ \oplus^\perp F_\tau^- \oplus^\perp
(\tilde F_0 \otimes\reels^2)$ with  $F_e^+$ (and $(F_e^+)^*$ totally
isotropic, $F_\tau^+$, $F_\tau^-$ and $\tilde F_0$ non degenerate

For $R$ a formal curvature tensor matching $\g g$, as $R(x,y)=0$ for
$x\perp y$ (by (*)), $\g g$ is generated by the $R(x,y)$ for
$(x,y)\in F_e^+\times (F_e^+)^*$, (respectively $(x,y)\in
F_\tau^+\times F_\tau^+$, resp. $(x,y)\in F_\tau^-\times F_\tau^-$).
$R(x,y)$ acts only on $F_e^+\oplus (F_e^+)^*$ (respectively
$F_\tau^+$, resp. $F_\tau^-$).

For $(x,y)\in F_e^+\times (F_e^+)^*$, $z\in V_1\cap F_e$, $t\in
V_1\cap F_e$, one has $\langle R(x,y)z,t\rangle =\langle
R(z,t)x,y\rangle =0$, and similarly for $(x,y)\in F_e^+\times
(F_e^+)^*$, $z\in V_2\cap F_e$, $t\in V_2\cap F_e$, one has $\langle
R(x,y)z,t\rangle=0$. So we obtain: $\g g F_e \subset V_1 \cap
V_1^\perp \oplus V_2 \cap V_2^\perp \subset \ker(\theta_{\mathcal
V})$.

Recall that $\theta$ maps $W_1$ into $W_2$ and $W_2$ into $W_1$.

For $(x,y)\in F_\tau^+\times F_\tau^+$, $z\in F_\tau^+$, $t\in
F_\tau^-$, one has: $\langle \theta (R(x,y)z),t\rangle =\langle
R(x,y)z,\theta(t)\rangle=\langle R(z,\theta(t))x,y\rangle=0$ because
$z \perp \theta(t)$. So $\theta (R(x,y)F_\tau^+)\subset F_\tau^-
\cap (F_\tau^-)^\perp=\{0\}$. Similarly for $(x,y)\in F_\tau^-\times
F_\tau^-$, $\theta (R(x,y)F_\tau^-)=\{0\}$, so $\g g F_\tau
\subset\ker(\theta_{\mathcal V})$.

$\g g E \subset \ker(\theta_{\mathcal V})$ follows from the
preceding observations. As $\theta$ commutes with every element of
$\g{g}$, we will have as well: $\g{g}
 \im \theta =\g{g} \theta(E) \subset \theta \g{g} E
 =\{0\}$.
\end{proof}

\begin{corr}
Let $E$ be a metric indecomposable representation of the Berger algebra $\g g$ preserving the
decomposition $E=V_1\oplus V_2$ with $V_1$ or $V_2$ degenerate. For
${\mathcal V}=(E, V_1, V_2, V_1^\perp, V_2^\perp)$, one has:
$\theta_{\mathcal V}^2=0$.
\end{corr}
\begin{proof} Recall that in the metric indecomposable case
with $E=V_1 \oplus V_2$ where $V_1$ or $V_2$ is degenerate, one has
$E=F_e$. Suppose $\theta_{\mathcal V}^2$ is non zero. In this case
one can choose a non trivial supplementary space $A$ of $\ker \theta
\cap \im \theta$ in $\im \theta$. $A$ is also a supplementary space
of $\ker \theta$ in $\ker \theta + \im \theta$. Let us choose a
supplementary space $B$ of $\ker \theta + \im \theta$ in $E$. One
has: Because $A \subset \im \theta$, there exists $A'$ subset of $E$
such that $A=\theta A'$. For $a\in\g{g}$, $aA=a\theta A'=\theta a
A'=\{0\}$ by the preceding theorem because $a A'\subset \g{g} E$. So
$A$ is invariant for the action of $\g{g}$. $\ker \theta + B$ is a
supplementary space of $A$, which is also invariant by $\g{g}$,
because $\g{g}(\ker \theta + B)\subset \ker \theta \subset \ker
\theta + B$. So we obtain a new decomposition of $E$ into two
$\g{g}$-invariant spaces $A$ and $\ker \theta + B$. the action of
$\g{g}$ on $A$ is trivial. So the action of $\g g$ decomposes into
an exterior product along the decomposition $A \oplus (\ker \theta +
B)$, in contradiction to what we supposed.
\end{proof}

\end{document}